\documentclass[a4paper,11pt]{article}

\usepackage{amsmath, amssymb}
\usepackage[english]{babel}
  \usepackage{paralist}
  \usepackage{graphics} 
  \usepackage{epsfig} 
\usepackage{graphicx}
\usepackage{epstopdf}
 \usepackage[colorlinks=true]{hyperref}
 \usepackage{array}
\usepackage[latin1]{inputenc}
\usepackage{xcolor}
\usepackage{esint}
\usepackage{latexsym,amsfonts}
\usepackage{amsthm}

\usepackage{verbatim}

\usepackage{array}
\newtheorem{thm}{Theorem}

\newtheorem{cor}{Corollary}[section]

\newtheorem{lem}{Lemma}[section]

\newtheorem{prop}{Proposition}[section]

\theoremstyle{definition}

\newtheorem*{theorem*}{Theorem}

\newtheorem*{remark}{Remark}

\numberwithin{equation}{section}

\def\Re{\mathfrak{Re\:}} 

\newcommand{\abs}[1]{\left\vert#1\right\vert} 
\newcommand{\brck}[1]{\left[#1\right]} 
\newcommand{\pare}[1]{\left(#1\right)} 

\newcommand{\set}[1]{\left\{#1\right\}}

\newcommand{\A}{\mathcal A} 

\newcommand{\R}{\mathbb R}
\newcommand{\N}{\mathbb N}

\newcommand{\Z}{\mathbb Z}

\newcommand{\C}{\mathbb C}

\newcommand{\U}{\mathfrak{U}} 
\newcommand{\TU}{\widetilde{\mathfrak{U}}} 
\newcommand{\bigint}{\begin{picture}(10,10)
\put(-1,2){\line(1,0){10}}
\end{picture}\kern-14pt\int}

\def\Xint#1{\mathchoice
   {\XXint\displaystyle\textstyle{#1}}%
   {\XXint\textstyle\scriptstyle{#1}}%
   {\XXint\scriptstyle\scriptscriptstyle{#1}}%
   {\XXint\scriptscriptstyle\scriptscriptstyle{#1}}%
   \!\int}
\def\XXint#1#2#3{{\setbox0=\hbox{$#1{#2#3}{\int}$}
     \vcenter{\hbox{$#2#3$}}\kern-.5\wd0}}

\def\dashint{\Xint-}

\begin{document}

\title{On the asymptotic mean value property for planar $p$-harmonic functions}
\author{ANGEL ARROYO and  JOS\'{E} G. LLORENTE }
\date{Departament de Matem\`{a}tiques \\ Universitat Aut\`onoma de Barcelona \\ 08193 Bellaterra. Barcelona \\SPAIN \\
arroyo@mat.uab.cat \\jgllorente@mat.uab.cat }

\maketitle

\footnotetext{  \emph{Key words:} mean value property, $p$-harmonic
function. MSC2010: 31C05, 35B60, 31C45. Partially supported by
grants MTM2011-24606, MTM2014-51824-P and 2014 SGR 75.}

\begin{abstract}
We show that $p$-harmonic functions in the plane satisfy a nonlinear
asymptotic mean value property for $p>1$. This extends previous
results of Manfredi and Lindqvist for certain range of $p$'s.
\end{abstract}


\section{Introduction}

It is well known that harmonic functions in euclidean domains are
those continous functions satisfying the usual mean value property.
Actually, harmonic functions can  be characterized by the so called
asymptotic mean value property:
$$
u(x) = \dashint_{B(x,r)} \hspace{-0.3cm} u(y) dy + o(r^2 )
$$
as $r\to 0$.

It is a challenging problem to try to find similar characterizations
for solutions of other nonlinear differential operators such as the
$p$-laplacian. We recall that a function $u\in W^{1,p}_{loc}(\Omega
) $ is $p$-harmonic in a domain $\Omega \subset \R^d$ if it is a
weak solution of the $p$-laplace equation
$$
div ( |\nabla u |^{p-2}\nabla u ) = 0
$$


If $u\in C^2$ and $\nabla u(x) \neq 0$ then  direct computation
shows that
\begin{equation}\label{plapl}
\triangle_p u \equiv div ( |\nabla u |^{p-2}\nabla u ) = |\nabla u
|^{p-2} \Big [ (p-2)\frac{\triangle_{\infty}u}{|\nabla u|^2} +
\triangle u \Big ]
\end{equation}
where $\triangle_{\infty}$ is the so called infinity laplacian
which, for $C^2$ functions, is given by
$$
\triangle_{\infty}u = \sum_{i,j=1}^d u_{x_i , x_j}u_{x_i}u_{x_j}
$$
On the other hand it follows essentially from Taylor's formula that
\begin{eqnarray}
\lim_{r\to 0}\frac{1}{r^2} \Big [ \frac{1}{2}\big ( \sup_{B(x,r)}u +
\inf_{B(x,r)}u \big ) - u(x) \Big ] & =
\displaystyle \frac{\triangle_{\infty}u(x)}{2|\nabla u (x)|^2} \label{inflapl} \\
\lim_{r\to 0}\frac{1}{r^2} \Big ( \dashint_{B(x,r)}\hspace{-0.3cm}u
- u(x) \Big ) & = \displaystyle \frac{\triangle u (x)}{2(d+2)}
\label{lapl}
\end{eqnarray}
where $B(x,r)$ denotes the open ball centered at $x$ of radius $r$.
From (\ref{plapl}),(\ref{inflapl}) and (\ref{lapl}) it can be shown
that if $u\in C^2$, $\triangle_p u = 0$ and $\nabla u (x) \neq 0$
then
\begin{equation}\label{AMVP0}
u(x) = \frac{p-2}{ p+d } \cdot \frac{1}{2}\Big ( \sup_{B(x,r)}u +
\inf_{B(x,r)}u \Big ) + \frac{2+d}{p+d}\, \dashint_{B(x,r)}
\hspace{-0.2cm}u(y)dy + o(r^2 )
\end{equation}
as $r\to 0$. Since $p$-harmonic functions are $C^{1, \alpha}$ for
some $0 < \alpha < 1$ but not $C^2$ in general (\cite{Ur} ,
\cite{Le}), it is not clear whether (\ref{AMVP0}) should be  true in
the general case. However, in \cite{PaMaRo} the authors proved that
if $u\in C(\Omega ) \cap W^{1,p}_{loc}(\Omega ) $ then $u$ is
$p$-harmonic in $\Omega $  if and only if  (\ref{AMVP0}) holds in a
weak(viscosity) sense. From \cite{JuLiMa} and \cite{PaMaRo}, it
follows that if $u$ is continous and satisfies (\ref{AMVP0}) in
classical sense then $u$ is $p$-harmonic.

\

More information is available when $d=2$. If $u$ is $p$-harmonic in
a planar domain  then the complex gradient $\partial u =
\frac{1}{2}(u_x - iu_y )$ is a quasiregular mapping( so the critical
points are isolated, unless $u$ is constant) and $u$ is $C^{\infty}$
in $\{ \nabla u \neq 0 \}$ (see \cite{BOJ-IWAN}, \cite{IWAN-MAN}).
Lindqvist and Manfredi have recently proven that in the plane
$p$-harmonicity is equivalent to the asymptotic mean value property(
in classical sense)  for a certain range of $p$'s. Hereafter we
denote by $D(x,r)$ the open disc of center $x$ and radius $r$.

\begin{theorem*} (\cite{LIND-MAN})
Let $\Omega \subset \R^2$ be a domain and let $1 < p < p_0 =
9.52...$. Then $u\in C(\Omega ) \cap W^{1,p}_{loc}(\Omega)$ is
$p$-harmonic in $\Omega$ if and only if the asymptotic expansion
\begin{equation} \label{AMVP}
        u(x) = \frac{p-2}{p+2}\cdot  \frac{1}{2} \Big (  \sup_{D(x,r)}u + \inf_{D(x,r)}u \Big )  +
        \frac{4}{p+2} \dashint_{D(x,r)} u(y) dy + o(r^2)
\end{equation}
holds at each $x\in \Omega$, as $r\to 0$.
\end{theorem*}

Our main result is an extension of Lindqvist and Manfredi's theorem
to the whole range of $p$'s.

\begin{thm}
Let $\Omega \subset \R^2$ be a domain and let $1<p<\infty$. Then
$u\in C(\Omega)\cap W^{1,p}_{loc}(\Omega)$ is $p$-harmonic in
$\Omega$ if and only if the asymptotic expansion (\ref{AMVP})
holds at each $x\in\Omega$, as $r\to 0$.
\end{thm}

It is enough to prove that a $p$-harmonic function satisfies
(\ref{AMVP}) at a critical point. Indeed, by the previous comments,
continous functions satisfying (\ref{AMVP}) are $p$-harmonic and, on
the other hand,  $p$-harmonic functions satisfy (\ref{AMVP}) at
noncritical points. Therefore we will focus on the local behavior of
a $p$-harmonic function  around a critical point of multiplicity
$n$. As in \cite{LIND-MAN}, the method exploits the power series
expansion of the complex gradient in the hodographic plane that was
obtained in \cite{IWAN-MAN}.

\section{The hodographic representation of $u$}

Let $u$ be a $p$-harmonic function with a critical point of
multiplicity $n$ at the origin and let
\begin{equation*}
    \partial u(z) = \frac{1}{2}(u_x-iu_y)
\end{equation*}
be the complex gradient of $u$. We can represent $\partial u(z) =
(\chi(z))^n$ where $\chi$ is quasiconformal in a neighborhood of the
origin and $\chi(0) = 0$ (see \cite{BOJ-IWAN}). In the hodographic
plane, $\xi=\chi(z)$, and, according to \cite{IWAN-MAN}, the inverse
of $\chi$ is given by
\begin{equation}\label{Hformula}
    z = H(\xi) = \sum_{k=n+1}^\infty \brck{ A_k \pare{\frac{\xi}{\abs{\xi}}}^k + \varepsilon_k \overline{A_k}
    \pare{\frac{\overline{\xi}}{\abs{\xi}}}^k}
    \pare{\frac{\xi}{\abs{\xi}}}^{-n} \abs{\xi}^{\lambda_k}
\end{equation}
where $A_k \in \C$, $A_{n+1} \neq 0$ and
\begin{equation}\label{sumAk}
\sum_{k=n+1}^\infty k \abs{A_k}^2<\infty
\end{equation}
Furthermore,
\begin{equation}\label{eps-lamb}
    \varepsilon_k = \frac{\lambda_k +n -k}{\lambda_k +n +k} \, \,  , \, \, \, \, \, \, \lambda_k
    = \frac{1}{2} \pare{ \sqrt{4k^2 (p-1) + n^2 (p-2)^2} -np}.
\end{equation}
From (\ref{eps-lamb}) it is easy to check that
\begin{equation}\label{eps-lamb bound}
0<\lambda_k<\frac{k^2-n^2}{n} \, \, , \, \, \, \, \, \,
\abs{\varepsilon_k}<\frac{k-n}{k+n}
\end{equation}.

Equation (\ref{Hformula})
 can be interpreted as the "hodographic
representation" of the point $z=x+iy$ near the origin.

Therefore, if $\xi = r e^{i\theta}$, we can rewrite (\ref{Hformula})
as
\begin{equation}\label{Hpolar}
    H(re^{i\theta}) = e^{-in\theta} \sum_{k=n+1}^\infty r^{\lambda_k} \varphi_k(\theta)
\end{equation}
where
\begin{equation*}
    \varphi_k(\theta) = A_k e^{ik\theta} + \varepsilon_k \overline{A_k} e^{-ik\theta}
\end{equation*}
for each $k=n+1,n+2,\ldots$ We can split $H(\xi)$ into its real and imaginary parts, i.e., $H(\xi) = \tilde{z}(\xi) = \tilde{x}(\xi) +i\tilde{y}(\xi)$.

Lets denote by $\tilde{u}$ the hodographic representation of $u$,
i.e.,
\begin{equation}\label{utilde}
    \tilde{u}(\xi) = (u\circ H)(\xi).
\end{equation}
Moreover, we can easily write $\partial u(z)$ in terms of $\xi$:
\begin{equation*}
    \partial u(z) = \partial u(H(\xi))=\xi^n
\end{equation*}
or, equivalently, if $\xi=re^{i\theta}$,
\begin{equation}\label{polargrad}
    \left\{
    \begin{array}{l}
    u_x = 2r^n \cos(n\theta) \\
    u_y = -2r^n \sin(n\theta)
    \end{array}
    \right.
\end{equation}
\\

\begin{prop}\label{HODOREP}
Let $u$ be a $p$-harmonic function  with a critical point of order
$n$ at $z=0$.  Then the hodographic representation $\tilde{u}$ has
the following power series expansion in a neighborhood of $\xi=0$:
\begin{equation*}
    \tilde{u}(\xi) = u(0) + 4 \sum_{k=n+1}^\infty \mu_k \abs{\xi}^{n+\lambda_k} \Re\set{ A_k \pare{\frac{\xi}{\abs{\xi}}}^k },
\end{equation*}
where $\displaystyle{ \mu_k = \frac{\lambda_k}{\lambda_k +n +k} }$.
Moreover, $0 \leq \mu_k <1-\frac{n}{k}$.
\end{prop}

\begin{proof}
Using (\ref{polargrad}), we compute $\tilde{u}_r$:
\begin{equation*}
    \tilde{u}_r = (u\circ H)_r =  u_x \tilde{x}_r + u_y \tilde{y}_r = 2r^n \brck{ \tilde{x}_r \cos(n\theta) - \tilde{y}_r \sin(n\theta) }
\end{equation*}
the expression in brackets being equal to $\Re\set{ e^{in\theta}
H_r}$. Replacing (\ref{Hpolar}) in the previous equation we get
$$
    \tilde{u}_r  =  2r^n \Re\set{ e^{in\theta} H_r} =
     4 \sum_{k=n+1}^\infty \lambda_k (1+\varepsilon_k) r^{n+\lambda_k-1} \Re\set{A_k
    e^{ik\theta}}.
$$
Integrate with respect to $r$ to complete the proof. The bound on
$\mu_k$ follows from (\ref{eps-lamb bound}).
\end{proof}

\begin{remark}
From now on, we can assume without loss of generality that $u(0) = 0$.
\end{remark}

\section{Quantitative injectivity estimates for the first term of $H$}

We define the mapping $\A(\xi)$ as the first term in the power series expansion of $H$:
\begin{equation}\label{Atilde}
            \A(\xi)  =  \brck{ A_{n+1} \pare{\frac{\xi}{\abs{\xi}}}^{n+1} + \varepsilon_{n+1} \overline{A_{n+1}}
    \pare{\frac{\overline{\xi}}{\abs{\xi}}}^{n+1}}
    \pare{\frac{\xi}{\abs{\xi}}}^{-n} \abs{\xi}^{\lambda_{n+1}}
\end{equation}
We define $\TU(\xi)$ to be the first term in the power series
expansion of $\tilde{u}$,
\begin{equation}\label{TUgothic}
    \TU(\xi) = 4\mu_{n+1} \abs{\xi}^{n+\lambda_{n+1}} \Re\set{A_{n+1} \pare{\frac{\xi}{\abs{\xi}}}^{n+1} }.
\end{equation}

For simplicity, we will use hereafter the notations $a \lesssim b$ (
resp. $a \approx b$) to indicate that $a \leq C b $ (resp. $C^{-1}a
\leq b \leq Ca$) for some positive constant $C$ independent of $a$
and $b$.

\begin{lem}\label{LEMMA03}
The following estimates hold in a neighborhood of $\xi = 0$:

\vspace{-0.2cm}
 \begin{eqnarray}
                \abs{\widetilde{u}(\xi) -\widetilde{\U}(\xi ) } & \lesssim &
                \abs{\xi}^{n+\lambda_{n+2}} \vspace{0.35cm} \label{est-error}\\
                \abs{H(\xi) - \A(\xi)} & \lesssim &
                \abs{\xi}^{\lambda_{n+2}}\vspace{0.35cm}\label{est-H-A}\\
                \abs{\A(\xi)} \approx \abs{H(\xi)} & \approx &
        \abs{\xi}^{\lambda_{n+1}} \label{estimates}
\end{eqnarray}
\end{lem}

\begin{proof}

From (\ref{sumAk}), ( \ref{eps-lamb bound}) and the fact that $0\leq
\mu_k < 1$  we get in particular that the sequence $(A_k )$ is
bounded and that $|\epsilon_k | < 1 $ for all $k$. Since $(\lambda_k
)$ is increasing, (\ref{est-error}), (\ref{est-H-A}) and
(\ref{estimates}) follow from the estimate
\begin{equation}\label{estimaxi}
\sum_{k=n+2}^{\infty}|\xi |^{\lambda_k} = O( |\xi |^{\lambda_{n+2}})
\end{equation}
Now an elementary computation shows that there is $C= C(p) >0$ such
that  $\displaystyle \lambda_k - \lambda_{n+2} \geq C(k - (n+2)) $
for all $k\geq n+2$.This implies (\ref{estimaxi}) and proves the
lemma.
\end{proof}





Now, we study the behavior of $\A$ and we give an injectivity
estimate. For this purpose, we will need the help of the following
elementary lemma whose proof is omitted.

\begin{lem}
Let $\rho >0$, $\lambda > 0$ and $t\in\R$. Then for all $k\in \N$,
                \begin{equation}\label{ineq-A}
                         \abs{\rho e^{ikt}-1} \leq k \abs{\rho e^{it}-1}.
                \end{equation}
Furthermore, if $\Lambda> 1$ and if $\Lambda^{-1} \leq \rho \leq
\Lambda$ then there is a constant $C = C(\lambda,\Lambda) > 0$ such
that
                \begin{equation}\label{ineq-B}
                        \abs{\rho^\lambda e^{it} - 1} \geq C \rho^{\lambda-1} \abs{\rho e^{it} - 1}.
                \end{equation}

\end{lem}




\begin{lem}\label{LEMMA02-1}
The mapping $\A : \C \to \C$ is bijective and  satisfies
\begin{equation}\label{ineq-Ca}
        \abs{\A(\xi) - \A(\zeta)} \geq C \abs{ \abs{\xi}^{\lambda_{n+1}-1} \xi - \abs{\zeta}^{\lambda_{n+1}-1} \zeta }
\end{equation}
where $C=(1-(2n+1)\abs{\varepsilon_{n+1}})\abs{A_{n+1}}$.
\end{lem}

\begin{proof}

First, we observe that from (\ref{eps-lamb bound}) for $k= n+1$ we
obtain that $\displaystyle 0 < \lambda_{n+1} < 2 + \frac{1}{n}$ and
that
\begin{equation} \label{epsi n+1}
\abs{\varepsilon_{n+1}}
<\frac{1}{2n+1}
\end{equation}
We  show first that $\A$ is surjective.  We write $\lambda \equiv
\lambda_{n+1}$, $\epsilon \equiv \epsilon_{n+1}$ and $A \equiv
A_{n+1}$. Then
$$
\A ( r e^{i\theta}) = r^{\lambda} e^{i\theta}\big ( A + \epsilon
\overline{A} e^{-i2(n+1)\theta} \big )
$$
Assume, for simplicity, that $A = 1$. Then we can write
$$
\A(re^{i\theta}) = r^{\lambda} |1 + \epsilon \, e^{-i2(n+1)\theta}|
e^{if(\theta )}
$$
where $\displaystyle f(\theta ) = \theta + \arg (1 + \epsilon \,
e^{-i2(n+1)\theta} )$ and

\begin{equation}\label{mperiodic}
        m(\theta) = |1 + \epsilon \, e^{-i2(n+1)\theta}| = \sqrt{ 1 + \varepsilon^2 + 2\varepsilon\cos(2(n+1)\theta) }.
\end{equation}

 To prove that $\A$ is surjective, let $w
= se^{it}\in \C$ such that $w\neq 0$ (if $w = 0$ it is obvious that
$\A(0) = 0$). Since $f(0)=0$ and $f(2\pi ) = 2\pi $, by continuity
we can pick $k\in \Z$ and $\theta \in [0, 2\pi ]$ such that $t +
2k\pi \in [0, 2\pi ]$ and $f(\theta ) = t + 2k\pi $. Then
$\displaystyle e^{if(\theta )} = e^{it}$. For that $\theta$, choose
$r>0$ so that
$$
\displaystyle r^{\lambda} m(\theta ) = s
$$
Then we have shown that $\A(re^{i\theta} ) = w$ so the
surjectiveness of $\A$ follows.


To finish the proof of the lemma, it is enough to prove
(\ref{ineq-Ca}), which is a quantitative form of injectiveness. By
(\ref{Atilde}),
\begin{eqnarray}
        \abs{\A(\xi)-\A(\zeta)} & \geq &  \abs{A_{n+1}} \abs{\abs{\xi}^\lambda \frac{\xi}{\abs\xi} - \abs{\zeta}^\lambda \frac{\zeta}{\abs{\zeta}}} - \nonumber \\ & - & \abs{A_{n+1}} \abs{\varepsilon} \abs{\abs{\xi}^\lambda \pare{\frac{\overline{\xi}}{\abs{\xi}}}^{2n+1} - \abs{\zeta}^\lambda \pare{\frac{\overline{\zeta}}{\abs{\zeta}}}^{2n+1}}  \label{ineq03}
\end{eqnarray}
Now apply (\ref{ineq-A}) with $\rho =
\abs{\displaystyle\frac{\xi}{\zeta}}^\lambda$,
$e^{it}=\displaystyle\frac{\xi/\zeta}{\abs{\xi/\zeta}}$ and
$k=2n+1$, and multiply both sides of the inequality by
$\abs{\zeta}^\lambda$. Then
\begin{equation*}
        \abs{ \abs{\xi}^\lambda \pare{\frac{\xi}{\abs{\xi}}}^{2n+1} - \abs{\zeta}^\lambda \pare{\frac{\zeta}{\abs{\zeta}}}^{2n+1} } \leq (2n+1) \abs{ \abs{\xi}^\lambda \frac{\xi}{\abs{\xi}} - \abs{\zeta}^\lambda \frac{\zeta}{\abs{\zeta}} }.
\end{equation*}
Replacing this expression in (\ref{ineq03}) we obtain
\begin{equation*}
        \abs{\A(\xi) - \A(\zeta)} \geq \abs{A_{n+1}} (1-(2n+1)\abs{\varepsilon}) \abs{ \abs{\xi}^\lambda \frac{\xi}{\abs{\xi}} - \abs{\zeta}^\lambda \frac{\zeta}{\abs{\zeta}} }.
\end{equation*}
so the proof is finished.
\end{proof}

\begin{lem}\label{LEMMA02-2}
Let $\Lambda >1$. Then there is a constant
$C=C(n,p,\Lambda,\abs{A_{n+1}}) >0$ such that for any $\xi$, $\zeta
\in \C$ with $\displaystyle \Lambda^{-1}\abs{\zeta} \leq \abs{\xi}
\leq \Lambda \abs{\zeta}$ we have
\begin{equation}\label{ineq-Cb}
        \abs{\A(\xi) - \A(\zeta)} \geq C \abs{\xi}^{\lambda_{n+1}-1} \abs{\xi - \zeta}.
\end{equation}
\end{lem}

\begin{proof}
Apply (\ref{ineq-B}) with $\rho =
\abs{\displaystyle\frac{\xi}{\zeta}}$ and
$e^{it}=\displaystyle\frac{\xi/\zeta}{\abs{\xi/\zeta}}$ and
multilply by $\abs{\zeta}^\lambda$ to obtain
\begin{equation}\label{ineq05}
        \abs{ \abs{\xi}^\lambda \frac{\xi}{\abs{\xi}} - \abs{\zeta}^\lambda \frac{\zeta}{\abs{\zeta}} } \geq C \abs{\xi}^{\lambda-1} \abs{ \xi - \zeta }
\end{equation}
Then the lemma follows from (\ref{ineq-Ca}) together with
(\ref{ineq05}).
\end{proof}

\section{The perturbation method}

Given $\xi$ in the hodographic plane, set $z = H(\xi )$, $\zeta =
\A^{-1}(z)$ and $w = H(\zeta )$. Then
$$
\xi = \chi (z) \, \, , \, \, \, \, \, \zeta = \chi (w) =
\A^{-1}(H(\xi ))
$$


\

 Since $|z| = |\A (\zeta ) | \approx |H(\zeta )| = |w|$ by
(\ref{estimates}) it follows from quasiconformality (\cite{Ahlfors})
that $\abs{\xi} = |\chi (z) |  \approx  | \chi (w)| = \abs{\zeta}$.
We recall that $u(z)=(\tilde{u}\circ\chi)(z)$ and that $\TU$ is
given by (\ref{TUgothic}). Following \cite{LIND-MAN}, define the
functions
\begin{equation}\label{Ugothic}
    \U(z) = (\TU \circ \A^{-1})(z)
\end{equation}

\begin{lem}\label{LEMMA04}
Let $\Lambda >1$. There is a constant $C = C(n,p, \Lambda ,
|A_{n+1}| )>0$ such that for any $\xi$, $\zeta \in \C$ with
$\Lambda^{-1} |\zeta | \leq |\xi | \leq \Lambda |\zeta |$ then
\begin{equation}\label{ineq-D}
        \abs{\TU(\xi) - \TU(\zeta)} \leq C |\xi |^n |\A (\xi ) - \A (\zeta )|.
\end{equation}
\end{lem}

\begin{proof}
From (\ref{TUgothic}), the fact that $0 \leq \mu_k < 1 $ if $k \geq
n$ and direct computation it follows that
\begin{equation}\label{ineq06}
        \abs{\TU(\xi) - \TU(\zeta)} \leq C |A_{n+1}|\abs{\xi}^{n+\lambda_{n+1}-1} \abs{\xi-\zeta}
\end{equation}
where $C = C(n,\Lambda) > 0$. Then the conclusion follows from Lemma
\ref{LEMMA02-2}.



\end{proof}

\begin{cor}\label{cor}
Let $\xi$ and $\zeta$ be as in the beginning of the section. Then
the following estimate holds in a neighborhood of $\xi = 0$:
\begin{equation}\label{ineqUgot}
 \abs{\TU(\xi) - \TU(\zeta)} \lesssim |\xi |^{n + \lambda_{n+2}}
\end{equation}
\end{cor}

\begin{proof}
Use the fact that $|\xi | \approx |\zeta |$, Lemma \ref{LEMMA04} and
estimate (\ref{est-H-A}).
\end{proof}

Now we are ready to prove the following singular expansion of a
$p$-harmonic function.

\begin{prop}\label{SINGULAREXPANSION}
Let $u$ be a $p$-harmonic function with a singularity of order $n$
at $z=0$ and $u(0)=0$. Then $u$ can be written as
\begin{equation*}
        u(z) = \U(z) + O(\abs{z}^{\frac{n+\lambda_{n+2}}{\lambda_{n+1}}})
\end{equation*}
in a neighborhood of $z=0$.
\end{prop}

\begin{proof}
By  (\ref{utilde}) and (\ref{Ugothic})  we can write
$$
 u(z)-\U(z) = \tilde{u}(\xi)-\TU(\zeta) = \TU(\xi)-\TU(\zeta)+\tilde{u}(\xi) -
 \TU(\xi).
$$

By (\ref{est-error}), (\ref{estimates}) and Corollary \ref{cor}  we
get

\begin{equation}\label{ineq08}
        \abs{u(z) - \U (z)} \lesssim \abs{\xi}^{n+\lambda_{n+2}}
        \approx |H(\xi )|^{\frac{n + \lambda_{n+2}}{\lambda_{n+1}}}
        = |z|^{\frac{n + \lambda_{n+2}}{\lambda_{n+1}}}
\end{equation}
so the proof is finished.

\end{proof}

\section{Proof of Theorem 1}

As before, we will write $\lambda$ and $\varepsilon$ instead of
$\lambda_{n+1}$ and $\varepsilon_{n+1}$, respectively. We can assume
without loss of generality that $A_{n+1}=1$. Then
\begin{equation*}
        \A(re^{i\theta}) = r^\lambda e^{-in\theta} (e^{i(n+1)\theta} + \varepsilon e^{-i(n+1)\theta})
\end{equation*}
and $\abs{\A(re^{i\theta})} = r^\lambda m(\theta)$, where $m(\theta
) $ is given by (\ref{mperiodic}). Furthermore
$$
\TU(re^{i\theta})=4 \mu r^{n+\lambda}\cos((n+1)\theta).
$$
where $\mu = \mu_{n+1}$.
\\

Denote by $D_R = D(0,R)$ the open disc centered at $0$ with radius
$R>0$ and define the "hodographic disc" $\widetilde{D_R}$ as
$\A^{-1}(D_R)$. Then, a point $re^{i\theta}$ of the hodographic
plane belongs to $\widetilde{D_R}$ if and only if
$\abs{\A(re^{i\theta})}<R$. Then, $\widetilde{D_R}$ can be described
, in polar coordinates, as
\begin{equation*}
        \widetilde{D_R} = \set { re^{i\theta} \ : \ r < \Big ( \frac{R}{m(\theta )} \Big)^{1/\lambda}
        }
\end{equation*}





Now we define the function $J(\zeta)$ as the absolute value of the
jacobian of $\A(\zeta)$. Computing $J(\zeta)$ in polar coordinates
we get

\begin{equation}\label{jacob}
        J(re^{i\theta}) = \lambda r^{2(\lambda-1)} \big ( 1 - (2n+1)\varepsilon^2 -2n\varepsilon \cos(2(n+1)\theta) \big ),
\end{equation}
(Observe that, since $\displaystyle |\epsilon | < (2n+1)^{-1}$,  the
expression in the right hand side of (\ref{jacob}) is positive).


\begin{lem}\label{LEMMA-UGOTHIC}
The $p$-harmonic function $\U(z)$ given by (\ref{Ugothic}) satisfies
the following properties, for small enough $R>0$:

 \begin{equation}\label{SU}
                            \sup_{D_R}\U + \inf_{D_R}\U = 0,
                    \end{equation}
 \begin{equation}\label{MU}
                            \int_{D_R} \U = 0.
                    \end{equation}

\end{lem}

\begin{proof}
By (\ref{Ugothic}), we need to study the behavior of $\TU(\xi)$ in
$\widetilde{D_R}$. Then, (\ref{SU}) is a direct consequence of the
symmetries of $\widetilde{D_R}$. To show (\ref{MU}), observe that,
by change of variables

\begin{equation}\label{change}
\int_{D_R}\U (z) dz  \, \,  =  \int_{\widetilde{D_R}} \TU (\zeta)
J(\zeta ) d\zeta
\end{equation}
and using polar coordinates in (\ref{change}) we get
\begin{equation}\label{polar}
\int_{D_R}\U (z) dz  =   4\mu \lambda \int_0^{2\pi}\int_0^{r(\theta
)} r^{n +3\lambda -1}\cos ((n+1)\theta ) j(\theta) dr d\theta
\end{equation}
where
$$
r(\theta ) = \Big ( \frac{R}{m(\theta)} \Big )^{1/\lambda} \, \, ,
\, \, \, \, j(\theta ) = 1 - (2n+1)\epsilon^2 -2n\epsilon \cos
(2(n+1)\theta )
$$
and $m(\theta)$ is given by (\ref{mperiodic}). Now (\ref{MU})
follows directly from (\ref{polar}) and the symmetry properties of
$m(\theta )$ and $j(\theta )$.


\end{proof}

\begin{lem}\label{LEMMA06}
The inequality
\begin{equation}\label{lambda2}
    \frac{n + \lambda_{n+2}}{\lambda_{n+1}} > 2
\end{equation}
holds for each $1 < p < \infty$ and each $n \geq 1$.
\end{lem}

\begin{proof}
From (\ref{eps-lamb}) and some computation it follows that
inequality (\ref{lambda2}) is equivalent to
\begin{equation}\label{ineq09}
    n(p+2) \sqrt{ n^2 p^2 + 16(n+1) (p-1) } > n^2 p^2 + (-2n^2 + 8n)
    p - (2n^2 + 8n).
\end{equation}

Now we distinguish two cases. If $n=1$ then (\ref{ineq09}) becomes
\begin{equation}\label{n=1}
    (p+2) \sqrt{ p^2 + 32 (p-1) } > p^2 + 6p - 10.
\end{equation}
If the right hand is negative then the inequality follows.
Otherwise, squaring the previous inequality we get
\begin{equation*}
    2p^3 + 7 p^2 + 10p -19 > 0.
\end{equation*}
which holds for each $p>1$ since the left-hand side is increasing in
$p$ and vanishes for $p=1$. This proves (\ref{ineq09}) when $n=1$.

Now assume $n\geq 2$ and observe that $\sqrt{ n^2 p^2 + 16(n+1)
(p-1) } \geq np$ for each $p>1$. Then  (\ref{ineq09}) would follow
if:
\begin{equation*}
    n^2 p (p+2) > n^2 p^2 + (-2n^2 + 8n)
    p - (2n^2 + 8n),
\end{equation*}
which is equivalent to
\begin{equation*}
    (2n - 4) p + n + 4 > 0,
\end{equation*}
and holds trivially if $n \geq 2$. This finishes the proof of the
lemma.


\end{proof}

\begin{proof}[\textbf{Proof of Theorem 1}]
As stated at the introduction, we only need to prove that planar $p$
-harmonic functions satisfy (\ref{AMVP}) since the converse is well
known. We also discussed there that (\ref{AMVP}) need only to be
checked at a critical point. Therefore, we can assume that $x=0$,
$u(0)=0$ and that $0$ is a critical point of $u$.

 Let $r>0$ be small enough. By Proposition \ref{SINGULAREXPANSION} and
Lemma \ref{LEMMA-UGOTHIC},
\begin{equation}\label{ineq10}
        \int_{D(0,r)}\hspace{-0.35cm}u =  O \big ( r^{\frac{n+\lambda_{n+2}}{\lambda_{n+1}}}\big) .
\end{equation}
and
\begin{equation}\label{ineq11}
         \frac{1}{2}\Big ( \sup_{D(0,r)} u + \inf_{D(0,r)} u  \Big ) = O \big (  r^{\frac{n+\lambda_{n+2}}{\lambda_{n+1}}} \big).
\end{equation}
Finally, combine (\ref{ineq10}), (\ref{ineq11}) and divide by $r^2$
to obtain that for any $\alpha \in \R$
\begin{equation}
        \frac{1}{r^2}\Big [  \alpha \big( \frac{1}{2}\sup_{D(0,r)}u + \frac{1}{2}\inf_{D(0,r)}u \big) + (1-\alpha)
        \dashint_{D(0,r)} u
         \Big ] = O \big (  r^{\frac{n+\lambda_{n+2}}{\lambda_{n+1}}-2}\big ) .
\end{equation}

By \ref{LEMMA06} the exponent of $r$ in the right-hand side is
strictly positive.  Therefore, taking limits as $r\rightarrow 0$, we
show that (\ref{AMVP}) holds at the origin and we conclude the
proof.
\end{proof}

\begin{remark}
The proof actually shows that if $x$ is a critical point of the
$p$-harmonic function $u$ then (\ref{AMVP}) still holds at $x$ if
the coefficients $\displaystyle (p-2)/(p+2) $ and $\displaystyle
4/(p+2)$  are replaced by  $\alpha$ and $1-\alpha$ for arbitrary
$\alpha$.
\end{remark}

\end{document}